\documentclass[12pt,reqno]{amsart}
\usepackage[letterpaper, total={6in, 8in}, left=1.25in, top=1.5in]{geometry}

\usepackage{amsmath}
\usepackage{amsfonts}
\usepackage{amssymb}
\usepackage{graphicx}
\usepackage{color}
\usepackage{hyperref}
\usepackage{xcolor}

\newtheorem{theorem}{Theorem}[section]
\newtheorem{thm}[theorem]{Theorem}
\newtheorem*{theorem*}{Theorem}
\newtheorem{lemma}{Lemma}[section]
\newtheorem{corollary}[theorem]{Corollary}
\newtheorem{proposition}{Proposition}[section]

\newtheorem{remark}[theorem]{Remark}

\def\Ric{\text{Ric}}

\def\a{\alpha}
\def\l{\lambda}

\def\ve{\varepsilon}
\def\p{\partial}

\def\R{\mathbb{R}}

\def\vp{\varphi}

\def\k{\kappa}

\def\sn{\operatorname{sn}}

\def\Ric{\operatorname{Ric}}

\def\Sect{\operatorname{Sect}}

\def\inj{\operatorname{inj}}

\numberwithin{equation}{section}

\begin{document}

\title[Robin heat kernel]{Robin heat kernel comparison on manifolds}

\author{Xiaolong Li}\thanks{The first author's research is partially supported by a start-up grant at Wichita State University}
\address{Department of Mathematics, Statistics and Physics, Wichita State University, Wichita, KS, 67260}
\email{xiaolong.li@wichita.edu}

\author{Kui Wang} \thanks{The research of the second author is supported by NSFC No.11601359} 
\address{School of Mathematical Sciences, Soochow University, Suzhou, 215006, China}
\email{kuiwang@suda.edu.cn}




\begin{abstract}
We investigate the heat kernel with Robin boundary condition and prove comparison theorems for heat kernel on geodesic balls and on minimal submanifolds. We also prove an eigenvalue comparison theorem for the first Robin eigenvalues on minimal submanifolds. 
This generalizes corresponding results for the Dirichlet and Neumann heat kernels. 
\end{abstract}

\maketitle

\section{Introduction}
Let $\Omega$ be an $m$-dimensional  compact Riemannian manifold with smooth boundary $\partial \Omega$. Let $o\in \Omega$, and  $H_\a(o,x,t)$ be the $o$-centered Robin heat kernel of $\Omega$, i.e. $H_\a (o,x,t)$ solves the heat equation
\begin{align}\label{heat equation}
    u_t(x,t)-\Delta u(x,t)=0, \quad (x,t)\in \Omega\times(0,\infty), 
\end{align}
with the Robin boundary condition
\begin{align}\label{robin boundary}
    \frac{\p u}{\p \nu}(x, t)+\a u(x,t)=0, \quad (x,t) \in \partial \Omega\times(0,\infty),
\end{align}
and the initial condition
\begin{align}\label{initial condition}
    u(x,0)=\delta_o(x).
\end{align}
Here $\nu$ denotes the outward unit normal vector filed on  $ \p \Omega$, $\a \in \mathbb{R}$ is called the Robin parameter, and initial condition \eqref{initial condition} is interpreted as that
\begin{align*}
    \lim_{t\to 0^+}\int_\Omega H_\a(o,x,t)\vp(x)\,dx=\vp(o),
\end{align*}
for every continuous function $\vp(x)$ on $\Omega$. 

Let $M^m(\k)$ be the $m$-dimensional simply-connected space form of constant sectional curvature $\k$. Denote by $\bar B_{\bar o}(R)$ the geodesic ball of radius $R$ centered at $\bar o$ in $M^m(\k)$ and by $\bar H_\a (\bar o,x,t)$ the $\bar{o}$-centered Robin heat kernel of $\bar B_{\bar o}(R)$. In other words, $\bar H_\a (\bar o,x,t)$ satisfies 
\begin{align}\label{RHK-model}
    \begin{cases}
     \p_t\bar H_\a (\bar o,x,t)-\Delta \bar H_\a(\bar o,x,t)=0,   & x\in \bar B_{\bar o}(R), t>0,\\
       \bar H_\a(\bar o,x,0)=\delta_{\bar o}(x), & x\in \bar B_{\bar o}(R),  \\
       \frac{\p}{\p \nu} \bar H_\a (\bar o, x,t)+\a \bar H_\a(\bar o, x,t)=0,& x\in \p  \bar B_{\bar o}(R) , t>0.
    \end{cases}
\end{align}
Since the metric on $\bar B_{\bar o}(R)$ is rotational invariant, so is system \eqref{RHK-model}. Then $\bar H_\a$ is radially symmetric in the space variable  by the uniqueness of the solution to system \eqref{RHK-model}. Here and thereafter we rewrite $\bar{H}_\a$ as $\bar H_\a(r_{\bar o}(x),t)$, where $r_{\bar o} (x)$ is the distance function from $\bar o$ on $M^m(\k)$. 

The Robin boundary condition  generates a global picture of the boundary value  problems. Indeed, the Neumann $(\a =0)$ and the Dirichlet ($\a \to \infty)$ boundary conditions are all special cases of the Robin boundary conditions. Hence, existing results on the Dirichlet, or Neumann heat kernels naturally motivate the investigation on the Robin heat kernel. The first main result of this article is the following comparison theorem for the Robin heat kernel. 
\begin{theorem}\label{thm1}
Suppose $M$ is an $m$-dimensional complete Riemannian manifold, $o\in M$ and $B_{o}(R)\subset M$ is the geodesic ball of radius $R$ centered at $o$. Let $\a>0$. Denote by $H_\a(o,x,t)$ be the o-centered Robin heat kernel of $B_{o}(R)$, and by $\bar H_\a(r, t)$ be the $\bar o$-centered Robin heat kernel of $\bar B_{\bar o}(R)$. 
\begin{itemize}
\item [(1)] If the Ricci curvature of $B_o(R)$ is bounded from below by $(m-1)\k$ , then
     \begin{align}\label{d1}
         H_\a(o,x,t)\ge \bar H_\a(r_o(x),t)
     \end{align}
     for all $x\in \Omega$ and $t>0$.
     \item [(2)] If the sectional curvature of $B_o(R)$ is bounded from above by $\k$  and $R < \inj(o)$, the injectivity radius of $o$, then
     \begin{align}\label{d11}
         H_\a(o,x,t)\le \bar H_\a(r_o(x),t)
     \end{align}
     for all $x\in \Omega$ and $t>0$.
\end{itemize}
Where $r_o(x)$ is the distance function from $o$ on $M$. 
\end{theorem}

\begin{remark}
When $\a=+\infty$, i.e.  $H_\a$  is Dirichlet heat kernel,  inequality \eqref{d1}  was proved under a lower bound on sectional curvature by Debiard, Gaveau and Mazet in \cite{DGM77};
when $\a=0$ or $\a=+\infty$, i.e. $H_\a$ is Neumann or Dirichlet heat kernel, estimates \eqref{d1} and \eqref{d11} were proven by Cheeger and Yau in \cite{CY81}. 
\end{remark}
\begin{remark}
It is known from \cite[Section 10]{Li12} that for $\a>0$ the Sturm-Liouville  decomposition of the Robin heat kernel on $\Omega$ is given by
\begin{align}\label{expr}
    H_\a(x, y, t)=\sum_{i=1}^{\infty} \exp{(-\l_{i, \a} t)}\vp_i(x)\vp_i(y)
\end{align}
with convergence absolute and uniform for each $t>0$.
Where $\l_{i, \a}$ are Robin eigenvalues of Laplacian on $M$ and $\vp_i(x)$ are corresponding eigenfunctions, defined as in \eqref{robin eigen}. Then as $t\to +\infty$, the expression \eqref{expr} and Theorem  \ref{thm1}
implies the following Cheng's type eigenvalue comparison for Robin eigenvalue: if $\Ric\ge (m-1)\k$ on $B_o(R)$, then 
\begin{align}\label{Ch-1}
    \l_{1, \a}(B_o(R)) \le \l_{1, \a}(\bar B_{\bar o}(R));
\end{align}
if $\operatorname{Sect} \le \k$ on $B_o(R)$ and $R < \inj(o)$, then
\begin{align}\label{Ch-2}
    \l_{1, \a}(B_o(R)) \ge \l_{1, \a}(\bar B_{\bar o}(R)).
\end{align}
Estimates \eqref{Ch-1} and \eqref{Ch-2}  were proved by Savo in \cite{Savo20}. Recently, the authors \cite{LW21} extended \eqref{Ch-1} and \eqref{Ch-2} to the first Robin eigenvalue of the $p$-Laplacian  for $p\in(1,+\infty)$.
\end{remark}

We set up some notations before stating the next theorem. Let $M$ be an immersed  submanifold of a complete Riemannian manifold $(N, g_N)$. Denote by $D_o(R)$ the extrinsic ball of radius $R$ centered at $o\in M$, i.e. the smooth connected component of $\{x\in M, d_N(o,x)\le R\}$ which contains $o$, where  $d_N$ is the distance function induced by the metric $g_N$. 

Our second main result is a Robin heat kernel comparison theorem for minimal submanifolds.
\begin{theorem}\label{thm2}
Let $\a>0$ and $M^m$ be an $m$-dimensional  minimally immersed  submainfold of $N^n$.  Let $D_o(R)$ be the extrinsic ball of radius $R$ centered at $o \in M$. Suppose  the sectional curvature of $N$ is bounded from above by $\k$. If $\k>0$, we assume further  that 
\begin{align}\label{main-ass}
    R\le \min\{i_N(o), \frac{1}{\sqrt{\k}}\arctan\frac{\a}{\sqrt{\k}}\}.
\end{align}
Then the Robin heat kernel $H_\a(o,x,t)$ of $D_o(R)$ satisfies
     \begin{align}\label{estH}
         H_\a(o,x,t)\le \bar H_\a(d_N(o,x),t)
     \end{align}
for all $x\in D_o(R)$ and $t>0$, where $i_N(o)$ is the injectivity radius of $N$ from $o$.
\end{theorem}
\begin{remark}
For Neumann and Dirichlet heat kernels, Theorem \ref{thm2} was first proved by Cheng, Li and Yau in \cite{CLY84} for space form ambient spaces $N$, and by Markvorsen in \cite{Mar86} for ambient spaces with sectional curvature bounded from above. Besides, for the heat kernel of the Bergmann metric on algebraic varieties, a similar comparison as inequality \eqref{estH} was proved by Li and Tian in \cite{LT95}.
\end{remark}

Next, let's turn to the eigenvalue problem with Robin boundary condition. 
Let $M$ be an $m$-dimensional  smooth compact Riemannian manifold with non-empty smooth boundary. We consider the following Robin eigenvalue problem 
\begin{align}\label{robin eigen}
    \begin{cases}
    -\Delta u=\l u, & x\in \Omega,\\
    \frac{\p u}{\p \nu}+\alpha u=0, &x\in  \p \Omega.
    \end{cases}
\end{align} The first Robin eigenvalue for Laplace operator, denoted by $\l_{1, \a}(\Omega)$ is the smallest number such that \eqref{robin eigen} admits a solution. Moreover, it can be characterized as
\begin{align}\label{eq 1.2}
   \l_{1, \a}(\Omega)=\inf\left\{\int_\Omega |\nabla u|^2 d\mu_g+\a\int_{\p \Omega} u^2 \, dA: u\in W^{1,2}(\Omega), \int_\Omega u^2\, d\mu_g =1\right\}, 
\end{align}
where $d\mu_g$ is the Riemannian measure induced by the metric $g$ and $dA$ is the induced measure on $\p \Omega$.

Recall that the classical eigenvalue comparison theorem of Cheng \cite{Cheng75} states that the first Dirichlet eigenvalue of a geodesic ball in an $m$-dimensional complete Riemannian manifold whose Ricci curvature is bounded from below by $(m-1)\k$ is less than or equal to that of a geodesic ball in a space form of constant sectional curvature $\kappa$, and that the reverse inequality holds if the Ricci lower bound is replaced by the sectional curvature upper bound $\Sect_g\leq\kappa$ and the radius of the geodesic ball is no larger than the injectivity radius at its center. For minimal submanifolds of spaces forms, Cheng's type comparison theorems for Dirichlet eigenvalues were obtained by  Cheng, Li and Yau \cite{CLY84}. By the expression \eqref{expr} of $H_\a$ and letting $t\to +\infty$, estimate \eqref{estH} yields 
\begin{corollary}\label{coro1}
With the same assumptions as in Theorem \ref{thm2}, we have
\begin{align*}
    \l_{1, \a}(D_o(R))\ge \l_{1, \a}(\bar B_{\bar o}(R)),
\end{align*}
where $\l_{1, \a}(D_o(R))$ is the first Robin eigenvalue of $\Delta_M$ on $D_o(R)$, and $\l_{1, \a}(\bar B_{\bar o}(R))$ is the first Robin eigenvalue of Laplacian on geodesic ball $\bar B_{\bar o}(R)$ in $M^m(\k)$.
\end{corollary}

In the present paper, we prove a more general eigenvalue comparison theorem of Cheng's type for the Robin eigenvalue on minimal submanifolds. 
Let $M$ be a smooth compact submanifold of $(N, g_N)$, and denote by $R$ the outer radius of $M$ defined by 
\begin{align}\label{def-R}
 R:=\inf_{p\in M} \sup_{x\in M} d_N(p,x),   
\end{align}
where $d_N$ is the distance function induced by the metric $g_N$. 
Our next result states that 
\begin{theorem}\label{thm3}
Let $M^m$ be an $m$-dimensional compact, connected and minimally immersed submanifold of $N^n$ ($m<n$) with smooth boundary   and the  outer radius $R$. Suppose the sectional curvature of the ambient space $N$ is bounded  from above by $\k$ and $\a>0$. If $\k>0$, we assume further that $R\le\min\{\frac{1}{\sqrt{\k}}\arctan\frac{\a}{\sqrt{\k}}, i_N(M)\}$. Then 
\begin{align}\label{main ineq}
    \l_{1, \a}(M)\ge \l_{1, \a}(\bar B_{\bar o}(R)).
\end{align}
Where   $i_N(M)=\inf_{x\in M} i_N(x)$ and $i_N(x)$ is the injectivity radius of $N$ from $x$.
Moreover, the equality holds if and only if $M$ is isometric to $\bar B_{\bar o}(R) $.
\end{theorem}
\begin{remark}
 Corollary \ref{coro1} is a special case of Theorem \ref{thm3}. In fact Corollary \ref{coro1} is a direct consequence of  heat kernel comparison \eqref{estH} as $t\to \infty$; while  in  Theorem \ref{thm3}, we use another proof by translating eigenfunctions and using  Barta’s inequality, see Section \ref{sect7}.
\end{remark}

In addition, we also prove the following heat kernel comparison theorem for K\"ahler manifolds. 
\begin{thm}\label{thm k1}
Let $(M^m, g, J)$ be a  K\"ahler manifold of complex dimension $m$ whose holomorphic sectional curvature is bounded from below by $4\k$ and orthogonal Ricci curvature is bounded from below by $2(m-1)\k$. 
 Let $B_{o}(R)\subset M$ be the geodesic ball of radius $R$ centered at $o$. Let $\a>0$ and  $H_\a (o,x,t)$ be the o-centered Robin heat kernel of $B_{o}(R)$.  Then
     \begin{align}
         H_\a(o,x,t)\ge \bar H_\a(r_o(x),t),
     \end{align}
where $\bar H_\a(r,t)$ is the Robin heat kernel of a metric ball of radius $R$ in the K\"ahler model of holomorphic
sectional curvature $4\k$.
\end{thm}
For quaternion K\"ahler manifolds, we prove
\begin{thm}\label{thm k2}
Let $(M^m, g, I, J, K)$ be a quaternion K\"ahler manifold of complex quaternion dimension $m$ whose scalar curvature is bounded from below by $4\k$ and orthogonal Ricci curvature is bounded from below by $16m(m+2)\k$. 
 Let $B_{o}(R)\subset M$ be the geodesic ball of radius $R$ centered at $o$. Let $\a>0$ and denote by $H_\a(o,x,t)$ the o-centered heat kernel on $B_{o}(R)$ with Robin boundary condition.  Then
     \begin{align}
         H_\a(o,x,t)\ge \bar H_\a(r_o(x),t),
     \end{align}
where $\bar H_\a(r_o(x),t)$ is the Robin heat kernel of a metric ball of radius $R$ in the quaternion  K\"ahler model of scalar  curvature $16m(m+2)\k$.
\end{thm}

\begin{remark}
In a recent paper \cite{BY22}, Baudoin and Yang proved that for  Dirichlet heat kernel, the same results as in Theorem \ref{thm k1} and Theorem \ref{thm k2} hold (i.e. the case of $\a=+\infty$), via  the study of the radial
parts of the Brownian motions.
\end{remark}
\begin{remark}
We note that  Riemannian model spaces (spheres and hyperbolic spaces) are not (quaternionic) K\"ahler manifolds, so Theorem \ref{thm k1} and Theorem \ref{thm k2} are sharper than  Theorem \ref{thm1} on (quaternionic) K\"ahler manifolds.
\end{remark}

This article is organized as follows. 
In Section 2, we study the Robin eigenvalue problem on geodesic balls in space of constant sectional curvature. 
In Section 3, we show the positivity of the Robin heat kernel using the maximum principle. 
Section 4 is devoted to the study of the Robin heat kernel on model spaces. 
The proofs of Theorems \ref{thm1}, \ref{thm2} and \ref{thm3} are given in Sections 5, 6 and 7, respectively. In Section 8, we present the proofs of Theorems \ref{thm k1} and \ref{thm k2}.

\section{Robin eigenvalue on geodesic balls in model spaces}
In this section, we set up the notation and recall some
facts on the eigenfunctions for the first Robin eigenvalue on geodesic balls  in space forms. 

The first Robin eigenvalue $\l_{1,\a}(\Omega)$ of the Laplacian is simple and its associated eigenfunction has a constant sign, thus can always be chosen to be positive. It follows from \eqref{eq 1.2} that $ \l_{1,\a}(\Omega)=0$ if $\a=0$,  $\l_{1,\a}(\Omega)>0$ if $\a>0$, and  $\l_{1,\a}(\Omega)<0$ if $\a<0$.

We denote by $\bar B_{\bar o}(R)$ the geodesic ball centered at $\bar o$ in the $m$-dimensional space form $M^m(\k)$ of constant sectional curvature $\k$ and by $\l_{1,\a}(\bar B_{\bar o}(R))$ the first Robin eigenvalue for $\bar B_{\bar o}(R)$ with Robin parameter $\a\in \mathbb{R}$. We write  $\l_{1,\a}(\bar B_{\bar o}(R))$ as $\bar\l_1$ for short. 

We collect some facts about the first Robin eigenfunctions associated with $\bar\l_1$. The eigenfunction associated to $\bar\l_1$ is radial due to the radial symmetry of $\bar B_{\bar o}(R)$. So we can choose a positive and radial function $u(r(x))$ on $\bar B_{\bar o}(R)$ as the first Robin eigenfunction associated to $\bar\l_1$, where $r(x)$ is the distance function from $\bar o$ in $M^m(\k)$. 
Then we deduce from the eigenvalue problem \eqref{robin eigen} that $u(r)$ solves the ODE initial value problem
\begin{align}\label{1dimu-1}
   \begin{cases}
   u''(r)+(m-1)\frac{\sn_\k'(r)}{\sn_\k(r)}u'(r)=-\bar\l_1 u(r), \quad r\in (0,R),\\
   u'(0)=0,\\
   u'(R)+\a u(R)=0.
    \end{cases}
\end{align}
Throughout the paper, we use the function $\sn_\k$ defined by
\begin{align*}
    \sn_\k(r)=\begin{cases}
    \frac{1}{\sqrt{\k}}\sin \sqrt{\k}r, & \k>0,\\
    r, &\k=0,\\
     \frac{1}{\sqrt{-\k}}\sinh \sqrt{-\k}r, & \k<0.
    \end{cases}
\end{align*}
It is easily seen from \eqref{1dimu-1} that $\bar{\l}_1$ is characterized by
\begin{align}\label{1-dim1}
  \bar{\l}_1=\inf\left\{\int_0^R |u'|^2 \sn_k^{m-1}\, dr+\a |u(R)|^2: u\in C^\infty([0,R]), \int_0^R u^2 \sn_\k^{m-1}\, dr=1\right\}.
\end{align}

We need the following properties of $u(r)$, see for example \cite[Lemma 8]{Savo20} and \cite[Proposition 2.1]{LW21}. 

\begin{proposition}\label{prop-u}
Let $\a>0$ and $u(r)$ be a positive first eigenfunction associated to $\bar{\l}_1$. Then
\begin{itemize}
    \item [(1)] $u'(r)<0$ on $(0,R]$.
    \item [(2)] $(\log u)'$ is monotone decreasing on $(0,R]$. Particularly, $u'(r)\ge -\a u(r)$ on $(0,R]$.
\end{itemize}
\end{proposition}

For $\k>0$, we have following lower bound for $\bar\l_1$, which will be used later.
\begin{lemma}\label{2lm2.1}
If $\a>0$, $\k>0$ and $\sqrt{\k}\tan(\sqrt{\k}R)\le \a$. Then
\begin{align}\label{pl}
    \bar\l_1\ge m\k.
\end{align}
\end{lemma}
\begin{proof}
Let $u(r)$ be the positive eigenfunction associated to $\bar\l_1$. Using Bochner formula, we estimate that
\begin{align*}
\frac 1 2 \Delta |\nabla u|^2=&|\nabla^2 u|^2+\Ric(\nabla u, \nabla u)+\langle \nabla \Delta u,\nabla u\rangle\nonumber\\
\ge &\frac{(\Delta u)^2}{m}++\Ric(\nabla u, \nabla u)+\langle \nabla \Delta u,\nabla u\rangle\nonumber\\
=&\frac{\bar\l_1^2 u^2}{m}+\Big((m-1)\k-\bar\l_1 \Big)|\nabla u|^2,
\end{align*}
where we used inequality $|\nabla^2 u|^2\ge \frac{(\Delta u)^2}{m}$ in the inequality, and  equation $\Delta u=-\bar \l_1 u$ and  $\Ric=(m-1)\k$ in the last equality. Integrating above inequality over $\bar B_{\bar o}(R)$ yields
\begin{align}\label{pl1}
    \int_{\p \bar B_{\bar o}(R)}\frac 1 2\frac{\p }{\p \nu} |\nabla u|^2\ge \frac{\bar\l_1^2 }{m}\int_{\bar B_{\bar o}(R)} u^2+\Big((m-1)\k-\bar\l_1\Big)\int_{\bar B_{\bar o}(R)} |\nabla u|^2.
\end{align}
Using ODE \eqref{1dimu-1}, we calculate
\begin{align}\label{pl2}
    \begin{split}
    \frac 1 2\frac{\p }{\p \nu} |\nabla u|^2=&u''u'=-\Big(\frac{(m-1)\operatorname{sn}'_\k}{\operatorname{sn}_\k} u'+\bar\l_1 u\Big) u'\\
=&\Big(-(m-1)\a\frac{\sqrt{\k}}{\tan(\sqrt{\k}R)}+\bar\l_1\Big)\a u^2\\
    \le & \a\Big(-(m-1)\k+\bar\l_1\Big) u^2
    \end{split}
\end{align}
 at $r=R$, where we used the assumption  $\sqrt{\k}\tan(\sqrt{\k}R)\le \a$ in the last inequality.
Combining  \eqref{pl1} with  \eqref{pl2}, we conclude
\begin{align}\label{pl3}
\begin{split}
  &\a\bar\l_1 \int_{\p \bar B_{\bar o}(R)} u^2+\bar\l_1\int_{\bar B_{\bar o}(R)} |\nabla u|^2\\
  \ge &\frac{\bar\l_1^2 }{m}\int_{\bar B_{\bar o}(R)} u^2+(m-1)\k\int_{\bar B_{\bar o}(R)} |\nabla u|^2+(m-1)\a\k\int_{\p \bar B_{\bar o}(R)} u^2.
  \end{split}
\end{align}
Recall from \eqref{eq 1.2} that 
\begin{align*}
   \a\int_{\p \bar B_{\bar o}(R)} u^2 +\int_{\bar B_{\bar o}(R)} |\nabla u|^2=\bar\l_1 \int_{\bar B_{\bar o}(R)} u^2,
\end{align*}
then inequality \eqref{pl3} gives
\begin{align*}
   \bar\l_1^2\int_{\bar B_{\bar o}(R)} u^2\ge \frac{\bar\l_1^2 }{m}\int_{\bar B_{\bar o}(R)} u^2+(m-1)\k\bar\l_1\int_{\bar B_{\bar o}(R)} u^2,
\end{align*}
proving the lemma.
\end{proof}

\section{Asymptotic and positivity of Robin  heat kernels}
In this section, we will recall some basic properties  of Laplace heat kernels on manifolds. 
It is well known that  Laplace heat kernels on manifolds are smooth in $\Omega\times \Omega \times \R_+$ and have an asymptotic expansion as $t\to 0$ and near the diagonal of the form
\begin{align}\label{asm-1}
   \frac{\exp(-\frac{d^2(x,y)}{4t})}{(4\pi t)^{\frac{n}{2}}}\Big(\sum_{i=0}^{\infty} a_i(x,y) t^i\Big)
\end{align}
by Minakshisundaram-Pleijel's construction, where $a_i(x,y)$ are smooth with $a_0(x,x)=1$ and $d(x,y)$ is the distance function on $\Omega$, see \cite[Formula 1.14]{CY81},  \cite[Sections 3 and 4 of Chapter VI]{Cha84}, \cite[Section 7.5]{Gri09}, \cite{Ber68}  \cite[Formula (1.1)]{Mc92} and \cite{Vas03}.  Similar as Dirichlet  and Neumann heat kernels, Robin heat kernel  is also positive when Robin parameter is positive, see \cite{GMN14,GMNO15}. For the  readers' convenience,  we give a direct proof of  the positivity of Robin heat kernels on manifolds by adopting Cheeger and Yau's arguments in \cite[Lemma 1.1]{CY81}.

\begin{lemma}\label{lm2.1} 
Let $H_\a(x,y,t)$ be the Robin heat kernel  defined by \eqref{heat equation}, \eqref{robin boundary} and \eqref{initial condition}. If $\a>0$, then we have 
\begin{align}\label{posi-H}
    H_\a(x,y,t)>0
\end{align}
for $t>0$ and  $x,y\in \Omega$.
\end{lemma}
\begin{proof}
In view of the asymptotic expansion \eqref{asm-1} for $H_\a(x,y,t)$, there exists  $\ve_0>0$ such that
\begin{align}\label{lm1-1}
 H_\a(x,y,t)>0   
\end{align}
for $d(x,y)<\ve_0$ and $0<t<\ve_0$. Now we fix $x$, and set
\begin{align}\label{lm1-2}
    h(y,t)=\frac 1 2 \Big( H_\a(x,y,t)-|H_\a(x,y,t)|\Big),
\end{align}
 and for  any $T>\ve_0$ denote
\begin{align*}
    \Omega_T:=\Omega\times(0,T]-B_x (\ve_0)\times (0,\ve_0].
\end{align*}
Clearly $h(y,t)\le 0$ in $\Omega\times(0,T]$,  $h(y,t)=0$ in $B_x (\ve_0)\times [0,\ve_0]$ due to  \eqref{lm1-1}, and $h(y,0)=0$ in $\Omega_T$ due to the initial condition. Then using integration by parts, we compute that
\begin{align}\label{lm1-3}
\begin{split}
  \int_{\Omega_{T}} \langle \nabla h, \nabla H_\a\rangle
 = &\int_{\ve_0}^T \int_\Omega \langle \nabla h, \nabla H_\a\rangle+\int_0^{\ve_0} \int_{\Omega\setminus B_x (\ve_0)} \langle \nabla h, \nabla H_\a\rangle\\
 = &\int_{\ve_0}^T \,dt\int_{\p \Omega} h\frac{\p H_\a}{\p \nu}\, d\sigma_y-\int_{\ve_0}^T \int_\Omega  h \Delta H_\a+\int_0^{\ve_0} \,dt\int_{\p \Omega} h\frac{\p H_\a}{\p \nu}\, d\sigma_y\\
 &-\int_0^{\ve_0} \,dt\int_{\p  B_x (\ve_0)} h\frac{\p H_\a}{\p \nu}\, d\sigma_y-\int_0^{\ve_0} \int_{\Omega\setminus  B_x (\ve_0)} h \Delta H_\a\\
 =& -\a \int_0^T \,dt\int_{\p \Omega} h H_\a\, d\sigma_y-\int_{\ve_0}^T \int_{B_x(\ve_0)}  h \p_t H_\a-\int_0^{T} \int_{\Omega\setminus  B_x (\ve_0)} h \p_t H_\a,   \end{split}
 \end{align}
 where we used the Robin boundary condition $\p H_\a /\p \nu =-\a H_\a$  and $h=0$ on $\p B_x(\ve_0)\times(0,{\ve_0})$ in the last equality, and $d\sigma_y$ denotes the induced measure on $\p\Omega$.
Observing from the definition \eqref{lm1-2} of $h$ that 
$$
hH_\a=h^2, \quad h\p_t H_\a= h\p_t h \text{\quad and \quad}\langle \nabla h, \nabla H_\a\rangle=|\nabla h|^2,
$$
then equality \eqref{lm1-3} becomes
\begin{align*}
\int_{\Omega_{T}} |\nabla h|^2 =-\a \int_0^T \,dt\int_{\p \Omega} h^2\, d\sigma_y-\frac 1 2\int_{\ve_0}^T \int_{B_x(\ve_0)} \p_t h^2-\frac 1 2 \int_0^{T} \int_{\Omega\setminus  B_x (\ve_0)} \p_t h^2.  
\end{align*}
Therefore, using $\a>0$, we have that
\begin{align*}
\int_{\Omega_{T}} |\nabla h|^2 \le& -\frac 1 2\int_{\ve_0}^T \int_{B_x(\ve_0)} \p_t h^2-\frac 1 2 \int_0^{T} \int_{\Omega\setminus  B_x (\ve_0)} \p_t h^2\\
=&\frac 1 2\int_{B_x(\ve_0)} h^2(y, {\ve_0})-\frac 1 2\int_{\Omega} h^2(y, T)\\
=&-\frac 1 2\int_{\Omega} h^2(y, T)\\
\le & 0,
\end{align*}
 where in the first inequality we used $h(y,0)=0$ for $y\in B_x(\ve_0)$, and in the last equality we used fact that 
 $h(y, {\ve_0})=0$ for $y\in B_{\ve_0}(x)$ due to \eqref{lm1-1}.
 Therefore
 \begin{align*}
     \int_{\Omega_{T}}|\nabla h|^2=0
 \end{align*}
implying $h\equiv 0$, so $H_\a(x,y,t)\ge 0$ on $\Omega_{T}$. Since $H_\a$ satisfies heat equation, thus
$$
H_\a(x,y,t)>0
$$
in $\Omega \times \Omega \times(0,\infty)$ by strong maximum principle. 
\end{proof}

\section{Robin heat kernel on model spaces}
In this section we will examine the general properties of the Robin heat kernel  on model spaces for later use.

Recall that $\bar H_\a(r_{\bar o}(x),t)$ denotes the $\bar o$-centered Robin Laplace heat kernel for geodesic ball $V(\bar o, \k , R)$ of radius $R$ centered at $\bar o$ in space form $M^m(\k)$. Then PDE \eqref{RHK-model} gives that $\bar H_\a$ satisfies
\begin{align}\label{1Dr}
   \frac{\p \bar H_\a}{\p t} (r,t)-\frac{\p^2 \bar H_\a}{\p r^2} (r,t)-(m-1)\frac{\sn_\k'(r)}{\sn_\k(r)}\frac{\p \bar H_\a}{\p r} (r,t)=0 
\end{align}
with initial value
\begin{align}\label{1Dr-i}
 \bar H_\a(r(x),0)=\delta_{\bar o}(x),   
\end{align}
and Robin boundary condition 
\begin{align}\label{1Dr-b}
     \frac{\p \bar H_\a }{\p r}  (R,t)+\a \bar H_\a(R,t)=0
\end{align}
for $t>0$.
Now we  rewrite Robin heat kernel $\bar H_\a$ on $V(\bar o, \k, R)$ as a function $\vp(s, t)$,
where
\begin{align}\label{def-s}
    s(r)=\begin{cases}
   \frac{1-\cos(\sqrt{\k} r)}{\k}, & \k>0,\\
    \frac{r^2}{2}, &\k=0,\\
    \frac{ \cosh (\sqrt{-\k}r)-1}{-\k}, & \k<0.
    \end{cases}
\end{align}
Clearly $s'(r)=\sn_\k(r)$.
Then equation \eqref{1Dr} becomes \begin{align}\label{1Ds}
    \p_t \vp(s,t)= \sn_\k^2(r) \vp''(s,t)+m\sn_\k'(r)\vp'(s,t)
\end{align}
for $(s,t)\in (0,s(R))\times(0,\infty)$, and the Robin boundary \eqref{1Dr-b} becomes
\begin{align}\label{1Ds-b}
     \sn_\k(R) \vp'(s(R),t)+\a \vp(s(R), t)=0
\end{align}
for $t>0$. Here and thereafter we denote  $\frac{\p^k \vp}{\p s^k}$ by $\vp^{(k)}$ for short. 
Differentiating  equation \eqref{1Ds} in $s$ twice yields
\begin{align}\label{1Ds-1}
    \p_t \vp'(s,t)= \sn_\k^2(r)\vp^{(3)}(s,t)+(m+2)\sn_\k'(r)\vp''(s,t)-\k m \vp'(s,t),
\end{align}
and
\begin{align}\label{1Ds-2}
    \p_t \vp''(s,t)= \sn_\k^2(r)\vp^{(4)}(s,t)+(m+4)\sn_\k'(r)\vp^{(3)}(s,t)-\k (2m+2) \vp''(s,t).
\end{align}
Differentiating Robin condition \eqref{1Ds-b} in $t$ and applying equations \eqref{1Ds} and \eqref{1Ds-1}, we have 
\begin{align}\label{K3}
     \sn_\k^3 \vp^{(3)}+\big((m+2)\sn_\k\sn_\k'+\a \sn_\k^2 \big)\vp''+m(\a\sn'_\k-\k\sn_\k) \vp'=0
\end{align}
at $s=s(R)$ and $r=R$.
\begin{lemma}\label{lm-dvp}
Let $\a>0$. Then 
\begin{align}\label{posi-Hr}
    \vp'(s,t)<0
\end{align}
for $s<s(R)$ and $t>0$.
\end{lemma}
\begin{remark}
Because $\p_r\bar H_\a(r,t)=\vp'(s,t)\sn_\k(r)$, so Lemma \ref{lm-dvp} gives
\begin{align}\label{p-dvp}
   \p_r \bar H_\a(r,t)<0
\end{align}
for $(r,t)\in (0,R]\times (0,\infty)$.
\end{remark}
\begin{proof}[Proof of Lemma \ref{lm-dvp}]
By asymptotic expansion formula \eqref{asm-1}, we have following expansion formula  
\begin{align}\label{asm-r}
\vp(s,t)\sim \frac{\exp (-\frac {s} {4t})}{(4\pi t)^{\frac{m}{2}}}\Big(\sum_{i=0}^{\infty} b_i(s) t^i\Big), 
\end{align}
 near $s=0$ as $t\to 0$,
where $b_i$ are smooth 
with $b_0(0)=1$. Thus we have
\begin{align*}
\vp'(s,t)\sim \frac{\exp (-\frac {s} {4t})}{(4\pi t)^{\frac{m}{2}}}\Big(-\frac{1}{4t}\sum_{i=0}^{\infty} b_i(s) t^i+\sum_{i=0}^{\infty} b_i'(s) t^i\Big),   
\end{align*}
and the dominate term is 
$$
-\frac{1}{4t}\frac{\exp (-\frac {s} {4t})}{(4\pi t)^{\frac{m}{2}}}b_0(s)
$$
as $t\to 0$, see also \cite[(2.13)]{CY81}. Thus there exists  $\ve_0>0$  such that
\begin{align}\label{lm3.1-1} 
\vp'(s,t)<0 
\end{align}
for all $(s, t)\in [0,\ve_0]\times(0,\ve_0)$.

 Let
\begin{align*}
    \vp_1(s,t)=e^{\k m t} \vp'(s,t),
\end{align*}
and
\begin{align*}
   \bar \vp_1(s,t)=\frac 1 2 \Big(  \vp_1 (s,t)+|\vp_1(s,t)|\Big),
\end{align*}
then PDE \eqref{1Ds-1} becomes
\begin{align}\label{lm3.1-2}
    \p_t \vp_1(s,t)=\sn_\k^2(r) \vp_1''(s,t)+(m+2)\sn_\k'(r) \vp_1'(s,t)
\end{align}
For $T>\ve_0$, denote
\begin{align*}
    I_T:=[0, s(R)]\times(0,T]-[0, \ve_0]\times (0,\ve_0]. 
\end{align*}
Using integration by parts, we have
\begin{align*}
&\int_{I_T}\vp_1'(s,t) \bar \vp_1'(s,t) \sn_\k^{m+2}(r)\, ds\, dt\\
= &\int_{\ve_0}^T \int_0^{s(R)} \vp_1'(s,t)\bar \vp_1'(s,t) \sn_\k^{m+2}(r)\, dsdt+\int_0^{\ve_0} \int_{\ve_0}^{s(R)} \vp_1'(s,t) \bar \vp_1'(s,t) \sn_\k^{m+2}(r)\, ds\,dt\\
 = &\int_0^T \bar \vp_1(s(R),t)\vp'_1(s(R),t) \sn_\k^{m+2}(R)\, dt-\int_{I_T} \vp_1''(s,t)  \bar \vp_1(s,t)\sn_\k^{m+2}(r)\, ds\, dt\\
 &-\int_{I_T} (m+2) \vp_1'(s,t)\bar \vp_1'(s,t)\sn_\k'(r)\sn_\k^{m}(r)\, ds\, dt.
  \end{align*}
Noticing from \eqref{1Ds-b} and  Lemma \ref{lm2.1} that  
$
    \vp_1(s(R),t)<0
$, so $\bar \vp_1(s(R),t)=0$.
Then we have
\begin{align*}
 &\int_{I_T}\vp_1'(s,t) \bar \vp_1'(s,t) \sn_\k^{m+2}(r)\, ds\, dt\\
 \le &-\int_{I_T} \vp_1''(s,t)  \bar \vp_1(s,t)\sn_\k^{m+2}(r)+(m+2) \vp_1'(s,t)\bar \vp_1'(s,t)\sn_\k'(r)\sn_\k^{m}(r)\, ds\, dt\\
 =&-\int_{I_T}  \bar \vp_1(s,t)(\vp_1)_t(s,t) \sn_\k^{m}(r) \, ds\, dt\\
  =&-\frac 1 2 \int_0^{s(R)}  \bar \vp^2(s,T) \sn_\k^{m}(r) \, ds,
  \end{align*}
 where we used  inequality \eqref{lm3.1-1} and equation \eqref{lm3.1-2}.
Since
$$
 \vp_1'(s,t) \bar \vp_1'(s,t)=(\bar \vp_1'(s,t))^2,
$$
then
\begin{align*}
 \int_{I_T} (\bar \vp_1'(s,t))^2 \sn_\k^{m+2}(r)\, ds\, dt\le 0
  \end{align*}
implying $\bar\vp_1(s,t) \equiv 0$, so 
\begin{align*}
   \vp_1(s,t)\le 0
\end{align*}
for $t>0$. Since $\vp_1(s,t)$ satisfies  equation \eqref{lm3.1-2}, then
$\vp'(s,t)<0$ for $t>0$ follows from the strong maximum principle. 
\end{proof}

\begin{proposition}\label{prop4.1}
Let $\a>0$ and $\l>0$. Suppose $u(r):[0,R]\to \mathbb{R}$ is a solution to \begin{align}\label{1dimu}
   u''(r)+(m-1)\frac{\sn_\k'(r)}{\sn_\k(r)}u'(r)=-\l u(r)
\end{align} 
in $(0,R)$ with $u'(0)=0$. Let
$ g(r)=m \frac{\sn_\k'(r)}{\sn_\k(r)}u'(r)+\l u(r)$. Then 
\begin{align}
      \lim\limits_{r\to 0} g(r)=\lim\limits_{r\to 0} g'(r)=0,
\end{align}
and
 \begin{align}\label{limg2}
    \lim_{r\to 0} g''(r)=-\frac{2\l (\l-\k m)}{m(m+2)}u(0).
\end{align}
\end{proposition}
 
 \begin{proof}
 As $r\to 0$, equation \eqref{1dimu} gives
\begin{align}\label{udd}
\lim_{r\to 0} u''(r)=-\frac{\l}{m} u(0),
\end{align}
then we  have 
\begin{align}\label{limg}
\lim_{r\to 0} g(r)=\lim_{r\to 0} m \frac{\sn_\k'(r)}{\sn_\k(r)}u'(r)+\l u(0)=m\lim_{r\to 0} u''(r)+\l u(0)=0.
\end{align}
Differentiating $g$ in $r$ and using equality $(\sn_\k')^2-\sn_\k\sn_\k''=1$, we have
\begin{align}\label{gd}
g'(r)=-\frac{m}{\sn^2_\k(r)}u'(r)+m \frac{\sn_\k'(r)}{\sn_\k(r)}u''(r)+\l u'(r),
\end{align}
then using L'Hopital's rule and $u'(0)=0$, we compute that
\begin{align*}
 \lim_{r\to 0} g'(r) =&  \lim_{r\to 0}\Big( \frac{\frac{- m}{\sn_\k(r)}u'(r)+m \sn_\k'(r)u''(r)}{\sn_\k (r)}\Big)\\
 =&m\lim_{r\to 0}\Big( \frac{-\k\sn^2_\k(r)-1 }{\sn_\k(r)\sn'_\k(r)}u''(r)+\frac{1 }{\sn_\k^2(r)}u'(r)+u'''(r)\Big)\\
  =&m\lim_{r\to 0}\Big( \frac{-2\k\sn^2_\k(r)- (\sn'_\k(r))^2 }{\sn_\k(r)\sn'_\k(r)}u''(r)+\frac{1 }{\sn_\k^2(r)}u'(r)\Big)+m\lim_{r\to 0} u'''(r)\\
   =&m\lim_{r\to 0}\Big( \frac{-\sn'_\k(r) }{\sn_\k(r)}u''(r)+\frac{1 }{\sn_\k^2(r)}u'(r)\Big)+m\lim_{r\to 0} u'''(r)\\
 =&-\lim_{r\to 0} g'(r)+m\lim_{r\to 0} u'''(r),
\end{align*}
which implies
\begin{align}\label{limgd}
    \lim_{r\to 0} g'(r)=\frac{m}{2}\lim_{r\to 0} u'''(r).
\end{align}
Differentiating ODE \eqref{1dimu} of $u$ in $r$ and  and using equality $(\sn_\k')^2-\sn_\k\sn_\k''=1$, we obtain
\begin{align}\label{1dimud}
    u'''(r)+(m-1)\frac{\sn_\k'(r)}{\sn_\k(r)}u''(r)- \frac{(m-1)}{\sn_\k^2(r)}u'(r)=-\l u'(r).
\end{align}
As $r\to 0$, equation \eqref{1dimud} gives
\begin{align}\label{limuddd}
  \lim_{r\to 0} u'''(r)=-\frac{m-1}{m}  \lim_{r\to 0} g'(r),
\end{align}
where we used equality \eqref{gd} and $u'(0)=0$. Combining \eqref{limuddd} with  \eqref{limgd}, we obtain
\begin{align}\label{limg1}
    \lim_{r\to 0} g'(r)=\lim_{r\to 0} u'''(r)=0.
\end{align}
Differentiating \eqref{gd} in $r$ yields
\begin{align}\label{gdd}
g''(r)=\frac{2m\sn_\k'(r)}{\sn_\k^3(r)}u'(r)-\frac{2 m}{\sn^2_\k(r)}u''(r)+m \frac{\sn_\k'(r)}{\sn_\k(r)}u'''(r)+\l u''(r).
\end{align}
Using  L'Hopital's rule, \eqref{udd}, \eqref{limg1} and \eqref{gdd}, we have
\begin{align*}
 \lim_{r\to 0} g''(r)=&-\frac{\l^2}{m}u(0)+m\lim_{r\to 0}\Big( u^{(4)}(r) -\k\frac{\sn_\k(r)}{\sn'_\k(r)}u'''(r)-\frac{2}{\sn_\k(r)\sn_\k'(r)}u'''(r)\\
 &+\frac{4}{\sn^2_\k(r)}u''(r)-\frac{2\k}{\sn_\k(r)\sn'_\k(r)}u'(r)-\frac{4\sn'_\k(r)}{\sn^3_\k(r)}u'(r)\Big)\\
 =&-\frac{\l^2}{m}u(0)+m\lim_{r\to 0}u^{(4)}(r)-2m\lim_{r\to 0}u^{(4)}(r)-2m\k\lim_{r\to 0}u''(r)\\
 &-2\lim_{r\to 0}(g''(r)-\l u''(r)-mu^{(4)}(r))\\
 =&-\frac{3\l^2}{m}u(0)+m\lim_{r\to 0}u^{(4)}(r)+2\k\l u(0)-2\lim_{r\to 0}g''(r),
\end{align*}
which is equivalent to 
\begin{align}\label{limgdd}
3 \lim_{r\to 0} g''(r)=-\frac{3\l^2}{m}u(0)+2\k\l u(0) +m\lim_{r\to 0}u^{(4)}(r). 
\end{align}
Differentiating equation \eqref{1dimud}, we have
\begin{align}\label{udddd}
    u^{(4)}(r)+(m-1) \frac{\sn'_\k}{\sn_\k}u'''(r)-
 \frac{2(m-1)}{\sn_\k^2}u''(r)-\frac{2(m-1)\sn_\k'}{\sn_\k^3}u'(r)=-\l u''(r).
\end{align}
Letting $r\to 0$,  \eqref{udddd} implies
\begin{align}\label{limudddd}
    \lim_{r\to 0}u^{(4)}(r)+\frac{m-1}{m} \lim_{r\to 0} g''(r)+\frac{m-1}{m}\frac{\l^2}{m}u(0)=\frac{\l^2}{m}u(0),
\end{align}
where we have used \eqref{udd} and \eqref{gdd}. Then equality
\eqref{limg2} follows from equalities \eqref{limgdd} and \eqref{limudddd}.
 \end{proof}

\begin{lemma}
Let $\a>0$. Suppose
$
  R\le  \frac{1}{\sqrt{\k}}\arctan\frac{\a}{\sqrt{\k}}$.
For any $t>0$, it holds
\begin{align}
  \vp''(0, t)> 0.  
\end{align}
\end{lemma}
\begin{proof}
Using $s'(r)=\sn_\k(r)$ and $\p_r \bar H_\a=-|\nabla \bar H_\a|$, we have
\begin{align*}
    \vp''(s,t)=&\frac{1}{\sn_\k^2(r)}( \p^2_{r}\bar H_{\a}(r,t)-\frac{\sn_\k'(r)}{\sn_\k(r)}\p_r\bar H_\a(r,t))\\
    =&\frac{1}{\sn_\k^2(r)}(\Delta \bar H_\a+m\frac{\sn'_\k(r)}{\sn_\k(r)}|\nabla \bar H_\a|).
\end{align*}
Denote by $\bar\l_{i,\a}$ Laplace Robin eigenvalues of geodesic ball $\bar B_{\bar o}(R)$ and by $\phi_i$ the associated orthogonal eigenfunctions, satisfying either $\phi_i(\bar o)=0$ or $\phi_i$ is radial (rewritten as $\phi_i(r_{\bar{o}}(x))$), see \cite[Lemma 7]{CLY84}. Then 
the Sturm-Liouville decomposition \eqref{expr} gives
\begin{align*}
    \bar H_\a(r_{\bar o} (x),t)=\sum_{i=1}^\infty e^{-\bar \l_{i,\a} t}\phi_i(\bar o)\phi_i(r_{\bar o}(x))=\sum_{\l} e^{-\l t}\phi_\l(0)\phi_\l(r),
\end{align*}
 where the summation is taken over all $\l$ such that $\phi_\l(0)\neq 0$. Then direct calculation gives 
\begin{align*}
    \Delta \bar H_\a+m\frac{\sn'_\k(r)}{\sn_\k(r)}|\nabla \bar H_\a|=\sum_{\l}e^{-\l t}\phi_\l(0)\Big(-\l \phi_\l(r)+m\frac{\sn'_\k(r)}{\sn_\k(r)}|\phi_\l'(r)|\Big).
\end{align*}
We choose $\phi_\l>0$. Noticing 
\begin{align*}
 -\l \phi_\l=\Delta\phi_\l= \phi_\l'' +(m-1)\frac{\sn_\k'}{\sn_\k} \phi_\l' 
\end{align*}
and
$\phi_\l'(0)=0$, we have $\phi_\l''(0)<0$, so $\phi_\l'(r)<0$ for $r$ near zero. Then we conclude
\begin{align*}
   \Delta \bar H_\a+m\frac{\sn'_\k(r)}{\sn_\k(r)}|\nabla \bar H_\a|=\sum_{\l}e^{-\l t}\phi_\l(0)\Big(-\l \phi_\l(r)-m\frac{\sn'_\k(r)}{\sn_\k(r)}\phi_\l'(r)\Big)
\end{align*}
for all $r$ near $0$. Therefore applying Proposition \ref{prop4.1}, we get
\begin{align*}
    \vp''(0,t)=&\lim_{r\to 0}\sum_{\l}e^{-\l t}\phi_\l(0)\frac{-\l \phi_\l(r)-m\frac{\sn'_\k(r)}{\sn_\k(r)}\phi_\l'(r)}{\sn_\k^2(r)}\\
    =&\sum_{\l}e^{-\l t}\phi^2_\l(0)\frac{\l(\l-\k m)}{m(m+2)}\\
    >&0,
\end{align*}
where we used Lemma \ref{2lm2.1} if $\k>0$. We complete the proof of  the lemma.
\end{proof}

\begin{lemma}
For $t>0$, $0<s<s(R)$, we have
\begin{align}
    \vp''(s,t)>0.
\end{align}
\end{lemma}
\begin{proof}
Recall from the asymptotic expansion \eqref{asm-r} for $\vp$ that
\begin{align*}
    \vp''(s,t)\sim (4\pi t)^{-m/2}\exp(-\frac{s}{4t})\Big(\frac{1}{16t^2}\sum_{i=0}^\infty b_i(s)t^i-\frac{1}{2t}\sum_{i=0}^\infty b'_i(s)t^i+\sum_{i=0}^\infty b''_i(s)t^i\Big)
\end{align*}
near $s=0$ as $t\to 0$, then there exists $\ve_0>0$  such that
\begin{align}\label{2dK}
\vp''(s,t)>0  
\end{align}
for all $(s, t)\in [0,\ve_0]\times(0,\ve_0)$.  Let
\begin{align*}
    \vp_2(s,t)=e^{(2m+2)\k t} \vp''(s,t),
\end{align*}
and
\begin{align*}
    \bar \vp_2(s,t)=\frac 1 2 \Big(  \vp_2(s,t)-|\vp_2 (s,t)|\Big).
\end{align*}
Then equation \eqref{1Ds-2} becomes
\begin{align}\label{lm4.3-2}
    \p_t\vp_2(s,t)=\sn_\k^2(r)\vp''_2(s,t)+(m+4)\sn_\k'(r)\vp_2'(s,t).
\end{align}
For any $T>\ve_0$, denote
\begin{align*}
    I_T:=[0, s(R)]\times(0,T]-[0, \ve_0]\times (0,\ve_0]. \end{align*}
Using integration by parts and $\vp_2(\ve_0,t)>0$ for $t<\ve_0$, we calculate that
\begin{align}\label{lm4.3-1}
\begin{split}
&\int_{I_T}\vp_2'(s,t) \bar \vp_2'(s,t) \sn_\k^{m+4}(r)\, ds\, dt\\
= &\int_{\ve_0}^T \int_0^{s(R)}\vp_2'(s,t) \bar \vp_2'(s,t)  \sn_\k^{m+4}(r)\, ds\, dt+\int_0^{\ve_0} \, dt\int_{\ve_0}^{s(R)} \vp_2'(s,t) \bar \vp_2'(s,t)  \sn_\k^{m+4}(r)\, ds\\
 = &\int_0^T\vp'_2(s(R),t) \bar \vp_2(s(R),t)  \sn_\k^{m+4}(R)\, dt-\int_{I_T} \bar \vp_2(s,t) \vp_2''(s,t)\sn_\k^{m+4}(r)\, ds\, dt\\
 &-\int_{I_T} (m+4) \bar \vp_2(s,t) \vp_2'(s,t)\sn_\k'(r)\sn_\k^{m+2}(r)\, ds\, dt.
 \end{split}
  \end{align}
Using equation \eqref{K3} and the assumption $R\le\frac{1}{\sqrt{\k}}\arctan\frac{\a}{\sqrt{\k}}$, we estimate that  
\begin{align*}
    \vp'_2 \bar \vp_2=&-\frac{\bar \vp_2}{\sn^3_\k}\Big(\big((m+2)\sn_\k\sn_\k'+\a \sn_\k^2 \big)\vp_2+e^{(2m+2)\k t} m(\a\sn'_\k-\k\sn_\k) \vp'\Big)\nonumber\\
    \le&-\frac{(m+2)\sn_\k\sn_\k'+\a \sn_\k^2}{\sn^3_\k}\bar\vp_2^2
\end{align*}
for $r=R$, where in the inequality we used $\bar \vp_2\le 0$ and $\vp'(s,t)<0$.
Then equality \eqref{lm4.3-1} yields
\begin{align*}
 &\int_{I_T}\vp_2'(s,t) \bar \vp_2'(s,t) \sn_\k^{m+4}(r)\, ds\, dt\\
 \le &-\int_{I_T} \bar \vp_2  \vp_2''\sn_\k^{m+4}- (m+4) \bar \vp_2 \vp_2'\sn_\k'\sn_\k^{m+2}\, ds\, dt\\
= &-\int_{I_T}  \bar \vp_2  \frac{\p_t \vp_2-(m+4)\sn_\k'\vp_2'}{\sn_\k^2}\sn_\k^{m+4}+(m+4) \bar\vp_2 \vp_2'\sn_\k'\sn_\k^{m+2}\, ds\, dt\\
 =&-\int_{I_T}  \bar \vp_2(s,t)\p_t \vp_2 (s,t) \sn_\k^{m+2}(r) \, ds\, dt\\
  =&-\frac 1 2 \int_0^{s(R)}  \bar \vp_2^2(s,T) \sn_\k^{m+2} \, ds,
  \end{align*}
 where we used  equation  \eqref{lm4.3-2} in the first equality, $\bar \vp_2(s,\ve_0)=0$ for $s<\ve_0$, and $\vp_2(s,0)=0$ for $s>\ve_0$ from the initial condition  in the last equality.
Since
$$
 \bar\vp_2'(s,t) \vp_2'(s,t)=(\bar \vp_2'(s,t))^2,
$$
we conclude
\begin{align*}
 \int_{I_T} (\bar\vp_2'(s,t))^2 \sn_\k^{m+4}(r)\, ds\, dt
  \le  0
  \end{align*}
implying $\bar \vp_2\equiv 0$, so 
\begin{align*}
   \bar \vp''(s,t)\ge 0
\end{align*}
for $s<s(R)$ and  $t>0$. Then $\vp''(s,t)>0$ follows from the strong maximum principle for the heat equation \eqref{1Ds-2}.
\end{proof}

\section{Proof of Theorem \ref{thm1}}
\begin{proposition}\label{prop-5.1}
Let $M$ be a compact manifold with smooth boundary, $o\in M$ and $H_\a(o,y,t)$ be the $o$-centered Robin heat kernel of $M$. Assume 
 $F(y,t)$ is a function defined on $M\times (0,\infty)$ satisfying $F(y,t)\ge 0$ for all $(y,t)\in M\times (0,\infty)$ and $ F(y,0)=\delta_o(y)$.
 \begin{itemize}
\item [(1)] If $\p_t  F-\Delta F\ge 0$ in $M\times (0,\infty)$, and $\p F/\p \nu+\a F\ge 0$ on $\p M\times (0,\infty)$. Then 
    \begin{align}\label{pr5.1-1}
    H_\a(o,y,t)\le  F(y,t)
\end{align}
in $M\times(0,\infty)$.
\item [(2)] If $\p_t F-\Delta F\le 0$ in $M\times (0,\infty)$, and $\p F/\p \nu+\a F\le 0$ on $\p M\times (0,\infty)$. Then 
    \begin{align}\label{pr5.1-2}
    H_\a(o,y,t)\ge  F(y,t)
\end{align}
in $M\times(0,\infty)$.
\end{itemize}
\end{proposition}

\begin{proof}[Proof of (1)]
Note from Duhamel’s principle that
\begin{align*}
    &F(y, t)-H_\a(o,y,t)\\
    =&\int_M F(x,t) H_\a(x,y,0) \,dx-\int_M  F(x,0) H_\a(x,y,t) \,dx\\
    =&\int_0^t \frac{\p}{\p \tau}\int_M F(x,\tau) H_\a(x,y,t-\tau) \,dx \, d\tau\\
        =&\int_0^t \int_M F_\tau(x,\tau) H_\a(x,y,t-\tau) \,dx \, d\tau-\int_0^t \int_M F(x,\tau) \p_\tau H_\a(x,y,t-\tau) \,dx \, d\tau,
\end{align*}
see for example \cite[Page 865]{LT95}.
Then using the assumptions 
$
\p_t  F-\Delta F\ge 0
$ in $M\times (0,\infty)$  we estimate that
\begin{align*}
    &F(y, t)-H_\a(o,y,t)\\
    \ge &\int_0^t \int_M \Delta_x F(x,\tau) H_\a(x,y,t-\tau) \,dx \, d\tau-\int_0^t \int_M F(x,\tau) \Delta_x H_\a(x,y,t-\tau) \,dx \, d\tau\\
 = &\int_0^t \int_{\p M} \frac{\p F}{\p \nu_x}(x,\tau) H_\a(x,y,t-\tau)-F(x,\tau)\frac{\p  H_\a}{\p \nu_x} (x,y,t-\tau) \,d\sigma_x \, d\tau\\
 \ge &\int_0^t \int_{\p M} -\a F (x,\tau) H_\a(x,y,t-\tau)+\a F(x,\tau)H_\a(x,y,t-\tau) \,d\sigma_x \, d\tau\\
 =&0,
\end{align*}
where we used the boundary condition $\p F/\p \nu+\a F\ge 0$ on $\p M\times (0,\infty)$ in the last inequality. 
This completes the proof of  \eqref{pr5.1-1}. 

The proof of \eqref{pr5.1-2} is similar, we omit the detail here.
\end{proof}
An immediately consequence of Proposition \ref{prop-5.1} is the following monotonicity result on Robin parameter for the Robin heat kernel.
\begin{corollary}
Let $M$ be a compact manifold with smooth boundary $\p M$.
\begin{itemize}
\item [(1)] Denote by $G(x,y,t)$, $H_\a(x,y,t)$ and $K(x,y,t)$ be the Dirichlet, Robin (with positive Robin parameter $\a$) and Neumann heat kernels respectively. Then 
\begin{align*}
    G(x,y,t)<H_\a(x,y,t)<K(x,y,t)
\end{align*}
in $M\times M\times (0,\infty)$.

\item [(2)] Robin heat kernel is monotone decreasing in Robin parameter. Namely for $0<\a_1<\a_2$, it holds
\begin{align*}
   H_{\a_2}(x,y,t)<H_{\a_1}(x,y,t)
\end{align*}
in $M\times M\times (0,\infty)$.
\end{itemize}
\end{corollary}
\begin{proof}[Proof of Theorem \ref{thm1}]
(1) Suppose $\Ric\ge (m-1)\k$ on $B_o(R)$, then we have 
\begin{align}\label{laplace-com1}
\Delta r_o(x)\le(m-1)\frac{\sn_\k'(r)}{\sn_\k(r)}
\end{align}
for all $x\in M\setminus\{o, C(o)\}$ by the Laplace comparison for distance function. We transplant Robin heat kernel $\bar H_\a(r_{\bar o},t)$ on the model space $V(\bar o,\k,R)$ to the geodesic ball $B_o(R)$ by
\begin{align*}
    F(x,t):=\bar H_\a (r_o(x),t), 
\end{align*}
for $x\in B_o(R)$ and $t>0$.
Clearly $F(x,t)>0$ for $t>0$ and $F(x,0)=\delta_o(x)$.
Using inequalities \eqref{p-dvp} and  \eqref{laplace-com1}, we estimate
\begin{align*}
   & \p_t F(x, t)-\Delta F(x, t)\\
    =&\p_t\bar H_\a(r_o(x), t)-\bar H_\a'(r_o(x), t)\Delta r_o(x)-\bar H_\a''(r_o(x), t)\\
    \le &\p_t \bar H_\a(r_o(x), t)-(m-1)\frac{\sn_\k'(r)}{\sn_\k(r)}\bar H_\a'(r_o(x), t)-\bar H_\a''(r_o(x), t)\\
    =&0
\end{align*}
for all $x\in M\setminus\{o, C(o)\}$, where the last equality follows from \eqref{1Dr}. By the boundary condition \eqref{1Dr-b}, it follows that  $F$ satisfies
\begin{align*}
\p_\nu F(x, t)  +\a  F(x, t) =\bar H_\a'(R, t)  +\a  \bar H_\a(R, t)=0
\end{align*}
for $x\in \p B_o(R)$. Since the cut locus is a null set, standard argument via approximation shows that $\bar H(r_o(x),t)$ satisfies
\begin{align*}
    \begin{cases}
    \p_t F(x, t)-\Delta  F(x, t)\le 0, & (x,t)\in B_o(R)\times(0,\infty),\\
   \p_\nu  F(x, t)  +\a F(x, t) =0, & (x,t)\in \p B_o(R)\times(0,\infty),
    \end{cases}
\end{align*}
in the distributional sense. It then follows from part (2) of Proposition \ref{prop-5.1} that $ F(x, t)\le H_\a(o,x,t)$, proving
\begin{align*}
  \bar H_\a(r_o(x), t)\le H_\a(o,x,t)  
\end{align*}
for $(x,t)\in B_o(R) \times(0,\infty)$.

(2) Suppose  $\operatorname{Sect}\le \k$ on $B_o(R)$. Then we have $\Delta r(x)\ge(m-1)\frac{\sn_\k'(r)}{\sn_\k(r)}$ for all $x\in M\setminus\{o, C(o)\}$ by the Laplace comparison for distance function. Same argument as in the proof of (1) show that
\begin{align*}
    \begin{cases}
    \p_t F(x, t)-\Delta F(x, t)\ge 0, & (x,t)\in B_o(R)\times(0,\infty),\\
   \p_\nu  F(x, t)  +\a  F(x, t) =0, & (x,t)\in \p B_o(R)\times(0,\infty),
    \end{cases}
\end{align*}
in the distributional sense. The desired estimate $\bar H_\a(r_o(x), t)=F(x,t)\ge H_\a (o,x,t) $ holds from part (1) of Proposition \ref{prop-5.1}.

\end{proof}

\section{Proof of Theorem \ref{thm2}}
Let $M$ be an $m$-dimensional compact minimal submainfolds of $N^n$. We denote by $\nabla^M$ and $\nabla^N$ the covariant derivative on $M$ and $N$ respectively, and by $\Delta_M$ and $\Delta_N$ the Laplace-Beltrami operator on $M$ and $N$ respectively. 
 We recall the following Laplace comparison for distance function $d_N(o,x)$, see for example \cite[Lemma 1]{Mar89} and
\cite[Lemma 2.1]{Palmer99}.
\begin{proposition}\label{PLC}
Let $M^m$ be an $m$-dimensional  minimally immersed submainfold of $N^n$. Suppose the sectional curvature of the ambient space $N$ is bounded by $\k$ from above. Let $p\in M$, and $\Omega_p(R)$ be any smooth connected component of $B_p(R)\cap M$. If $\k>0$, assume further that $R\le \min\{i_N(p), \frac{\pi}{\sqrt{\k}}\}$. Let $F:[0,R]\to \mathbb{R}$ be a smooth function with $F'\ge 0$.  Then 
\begin{align}\label{L-C}
    \Delta_M F(r) \ge\Big(F''(r)-\frac{\sn_\k'(r)}{\sn_\k(r)}F'(r)\Big)|\nabla^M r|^2+m F'(r)\frac{\sn_\k'(r)}{\sn_\k(r)}
\end{align}
on $\Omega_p(R)\setminus p$. Where $i_N(p)$ is the injectivity radius of $N$ from $p$ and $r(x):=d_N(o,x)$.
\end{proposition}

\begin{proof}[Proof of Theorem \ref{thm2}]
We transplant the Robin heat kernel $\bar H(r_{\bar o},t)$ on the model space $V(\bar o,\k,R)$ to the extrinsic geodesic ball $D_o(R)$ of radius $R$ centered at $o$ by
\begin{align*}
    v(x,t):=\bar H_\a (r,t)=\vp(s(r),t), 
\end{align*}
where  $s=s(r)$  defined by \eqref{def-s}.
Clearly $v(x,t)>0$ for $t>0$. Note that the extrinsic distance function $d_N(o,x)$ is asymptotic to the intrinsic distance function as $t\to 0$, hence $v(x,t)\to \delta_o(x)$.

Since the sectional curvature of the ambient space $N$ is bounded from above by $\k$ and $s'(r)=\sn_\k(r)>0$, then applying Proposition \ref{PLC} we have
\begin{align}\label{LCs}
     \Delta_M s(r) \ge \Big(s''(r)-\frac{\sn_\k'(r)}{\sn_\k(r)}s'(r)\Big)|\nabla^M r|^2+m s'(r)\frac{\sn_\k'(r)}{\sn_\k(r)}=m\sn_\k'(r).
\end{align}
Direct calculation gives
\begin{align}\label{thm2-eq1}
\begin{split}
   \p_t v- \Delta_M v=&\p_t \vp- \vp'(s,t) \Delta_M s-\vp''(s, t)|\nabla^M s|^2\\
    \ge &\p_t \vp - m\sn_\k'(r)\vp'(s,t)-\vp''(s,t) \sn^2_\k(r)|\nabla^M d_N(o,x)|^2\\
     \ge &\p_t \vp - m\sn_\k'(r)\vp'(s,t)-\vp''(s,t) \sn^2_\k(r)\\
    =&0,
    \end{split}
\end{align}
where we used inequality  \eqref{LCs} and $ \vp'(s,t)< 0$ in the first inequality, $ \vp''(s,t)>0$ and $| \nabla^M d_N(o,x)|\le| \nabla^N d_N(o,x)|=1$ in the second inequality, and the PDE \eqref{1Ds} of $\vp$ in the last equality. 

For $x\in \p M$, using $\vp'<0$  and $| \nabla^M d_N(o,x)|\le1$, we estimate
\begin{align*}
    \frac{\p v}{\p \nu}=&\vp'(s(R),t) \sn_\k(R)\langle \nabla^M d_N(o,x), \nu\rangle\nonumber\\
    \ge &\vp'(s(R),t) \sn_\k(R)\nonumber\\
    = &-\a \vp (s(R),t),
\end{align*}
where we used  boundary condition \eqref{1Ds-b} in the last inequality. Thus
\begin{align}\label{thm2-eq2}
   \frac{\p v}{\p \nu}+\a v\ge 0,   
\end{align}
for $(x,t)\in \p M\times (0,\infty)$.
In view of \eqref{thm2-eq1} and \eqref{thm2-eq2}, part (1) of Proposition \ref{prop-5.1} gives
$$
H_\a(o,x,t)\le v(x,t),
$$
proving Theorem \ref{thm2}.
\end{proof}

\section{Proof of Theorem \ref{thm3}}\label{sect7}

\begin{proof}[Proof of Theorem \ref{thm3}]
Let $\bar \l$ be the first Laplace Robin eigenvalue for the geodesic ball $V(\bar o,\k, R)$ in space form $M^m(\k)$, and $u(r_{\bar o}(x))$ be a positive corresponding eigenfunction satisfying the ODE \eqref{1dimu-1}. We rewrite $u(r)$ as $w(s)$, where  $s(r)$ is defined by \eqref{def-s}  again. Therefore equation \eqref{1dimu-1} gives
\begin{align}\label{thm1.3-eq1}
    \sn_\k^2(r) w''(s)+m\sn_\k'(r)w'(s)=-\bar \l_1 w(s)
\end{align}
for $s\in[0, s(R))$ with boundary condition
\begin{align}\label{thm1.3-eq2}
     \sn_\k(R) w'(s(R))+\a w(s(R))=0.
\end{align}
It follows from (1) of Proposition \ref{prop-u} and
$w'(s)\sn_\k(r)=u'(r)$ that
\begin{align}\label{3-0}
 w'(s)<0 , \quad s\in (0, s(R)]. 
\end{align}
Direct calculation gives
\begin{align}
    w''(s)=-\frac{1}{\sn_\k^2(r)}\Big(m \frac{\sn_\k'(r)}{\sn_\k(r)}u'(r)+\bar \l_1 u(r)\Big),
\end{align}
thus Lemma \ref{2lm2.1} and Proposition \ref{prop4.1} implies
\begin{align}\label{thm1.3-eq3}
    \lim_{s\to 0}w''(s)>0.
\end{align}
Differentiating the equation \eqref{thm1.3-eq1} in $s$ yields
\begin{align}\label{thm1.3-eq4}
  \sn_\k^3(r)w'''(s)=-(\bar \l_1-\k m)\sn_\k(r)w'(s).
\end{align}
Then we conclude from \eqref{thm1.3-eq3} and \eqref{thm1.3-eq4}  that
\begin{align}\label{thm1.3-eq5}
    w''(s)>0
\end{align}
for $s\in[0, s(R)]$. In fact, if  $w''(s_0)=0$ for some $s_0=s(r_0)\in [0, s(R)]$, then using equation \eqref{thm1.3-eq4} and $w'(s)<0$ we have $w'''(s_0)>0$. Therefore \eqref{thm1.3-eq5} holds by
\eqref{thm1.3-eq3}.

Now we transplant the function $w$ to $M$  by
$$
v(x):=w\Big(s(r(x))\Big)
$$
for $x\in M$, where $r(x)=d_N(o,x)$. Applying  Proposition \ref{PLC}, we estimate that
\begin{align}\label{thm3-6}
\begin{split}
    \Delta_M v(x)=& w'(s) \Delta_M s+w''(s)|\nabla^M s|^2\\
    \le & m\sn_\k'(r)w'(s)+w''(s) \sn^2_\k(r)|\nabla^M d_N(o,x)|^2\\
     \le & m\sn_\k'(r)w'(s)+w''(s) \sn^2_\k(r)\\
    =&-\bar \l_1 w(s),
    \end{split}
\end{align}
where we used inequality \eqref{3-0} in the first inequality,  inequality \eqref{thm1.3-eq5} and $|\nabla^M d_N(o,x)|\le 1$ in the second inequality and  equation \eqref{thm1.3-eq1} of $w$ in the last equality. Thus
\begin{align}\label{eqtest1}
    -\Delta_M v(x)\ge \bar \l_1 v(x) \text{\quad in\quad} M.
\end{align}
For $x\in \p M$, we estimate
\begin{align*}
    \frac{\p v(x)}{\p \nu}=&w'(s) \sn_\k(r)\langle \nabla^M r, \nu\rangle\nonumber\\
    \ge &w'(s) \sn_\k(r)\nonumber\\
    \ge &-\a w(s),
\end{align*}
where we used  \eqref{3-0} and $|\langle \nabla^M r, \nu\rangle|\le 1$ in the first inequality, and (2) of Proposition \ref{prop-u} in the last inequality. Thus 
\begin{align}\label{eqtest2}
   \frac{\p v(x)}{\p \nu}+\a v(x)\ge 0
\end{align}
on $\p M$. Then we conclude from inequalities \eqref{eqtest1}, \eqref{eqtest2} and Barta's inequality (see for example \cite[Theorem 3.1]{LW21}) that
$$
\l_{1,\a}(M)\ge \bar \l_1,
$$
proving inequality \eqref{main ineq}.
Moreover if the equality occurs,  $v(x)$ is a constant multiple of the first eigenfunction. In this case, all inequalities above hold as equalities. Hence $\nabla^M d_N(o,x)=\nu$ on $\p M$ and $|\nabla^M d_N(o,x)|=1$ in $M$. Therefore  $M$ is a minimal cone, implying $M$ is isometric to $\bar B_{\bar o}(R) $. We complete the proof of Theorem \ref{thm3}.
\end{proof}

\section{K\"ahler and Quaternion K\"ahler manifolds}
Recall from system \eqref{RHK-model} that the  $o$-centered Robin heat kernel of the metric ball $\bar B_{\bar o}(R)$ of K\"ahler model space  satisfies one dimensional equation
\begin{align}\label{ka-1}
   \p_t \bar H_\a(r,t)-\frac{\p^2 \bar H_\a}{\p r^2}(r,t) -\Big(\frac{(2m-2)\sn_\k'(r)}{\sn_\k(r)}+\frac{\sn_{4\k}'(r)}{\sn_{4\k}(r)}\Big)\frac{\p \bar H_\a}{\p r}(r,t)=0.
\end{align}
Set
$$
\vp(s,t)=\bar H_\a(r,t),
$$
for $s\in[0, s(R)]$, where
\begin{align*}
    s(r)=\int_0^ r \big(\sn_\k^{2m-2}(r)\sn_{4\k}(r)\big)^{\frac {1}{2m-1}}\, dr.
\end{align*}
Let 
$$
\eta(r)=s'(r)=\big(\sn_\k^{2m-2}(r)\sn_{4\k}(r)\big)^{\frac {1}{2m-1}},
$$
then equation \eqref{ka-1} becomes 
\begin{align}\label{ka-1Ds}
    \vp_t(s,t)= \eta^2(r) \vp''(s,t)+2m\eta'(r)\vp'(s,t)
\end{align}
for $(s,t)\in (0,s(R))\times(0,\infty)$. 
Differentiating equation \eqref{ka-1Ds} in $s$  yields
\begin{align}\label{ka-1Ds-1}
    \p_t \vp'= \eta^2\vp^{(3)}+(2m+2)\eta'\vp''-\frac{2m}{2m-1}\Big((2m+2)\k+\frac{2m-2}{2m-1}(\frac{\sn_\k'}{\sn_\k}-\frac{\sn_{4\k}'}{\sn_{4\k}})^2\Big)\vp'.
\end{align}
\begin{lemma}\label{lm-ka}
Let $\a>0$. Then 
\begin{align}
    \vp'(s,t)<0.
\end{align}
for $s<s(R)$ and $t>0$. 
\end{lemma}
\begin{proof}
Similarly as in the proof of Lemma  \ref{lm-dvp}, we choose  $\ve_0>0$  such that
\begin{align}\label{ka-lm3.1-1}
\vp'(s,t)<0 
\end{align}
for all $(s, t)\in [0,\ve_0]\times(0,\ve_0)$.

 Let
\begin{align*}
    \vp_1(s,t)=e^{\frac{2m\k(2m+2)}{2m-1}t} \vp'(s,t),
\end{align*}
and
\begin{align*}
   \bar \vp_1(s,t)=\frac 1 2 \Big(  \vp_1 (s,t)+|\vp_1(s,t)|\Big),
\end{align*}
then equation \eqref{ka-1Ds-1} becomes
\begin{align}\label{ka-lm3.1-2}
    \p_t \vp_1=\eta^2(r) \vp_1''+(2m+2)\eta'(r) \vp_1'-\frac{2m(2m-2)}{(2m-1)^2}(\frac{\sn_\k'}{\sn_\k}-\frac{\sn_{4\k}'}{\sn_{4\k}})^2 \vp_1.
\end{align}
For any $T>\ve_0$, denote
\begin{align*}
    I_T:=[0, s(R)]\times(0,T]-[0, \ve_0]\times (0,\ve_0]. 
\end{align*}
Using integration by parts, we have
\begin{align*}
&\int_{I_T}\vp_1'(s,t) \bar \vp_1'(s,t) \eta^{2m+2}(r)\, ds\, dt\\
= &\int_{\ve_0}^T \int_0^{s(R)} \vp_1'(s,t)\bar \vp_1'(s,t) \eta^{2m+2}(r)\, ds\, dt+\int_0^{\ve_0} \int_{\ve_0}^{s(R)} \vp_1'(s,t) \bar \vp_1'(s,t) \eta^{2m+2}(r)\, ds\, dt\\
 = &\int_0^T \bar \vp_1(s(R),t)\vp'_1(s(R),t) \eta^{m+2}(R)\, dt-\int_{I_T} \vp_1''(s,t)  \bar \vp_1(s,t)\eta^{2m+2}(r)\, ds\, dt\\
 &-\int_{I_T} (2m+2) \vp_1'(s,t)\bar \vp_1'(s,t)\eta'(r)\eta^{2m}(r)\, ds\, dt.
  \end{align*}
Noticing from the Robin boundary condition \eqref{1Ds-b} and  Lemma \ref{lm2.1} that  
$
    \vp_1(s(R),t)<0
$, so $\bar \vp_1(s(R),t)=0$.
Then using equation \eqref{ka-lm3.1-2}, we estimate
\begin{align*}
 &\int_{I_T}\vp_1'(s,t) \bar \vp_1'(s,t) \eta^{2m+2}(r)\, ds\, dt\\
 \le &-\int_{I_T} \vp_1''(s,t)  \bar \vp_1(s,t)\eta^{2m+2}(r)+(2m+2) \vp_1'(s,t)\bar \vp_1(s,t)\eta'(r)\eta^{2m}(r)\, ds\, dt\\
 =&-\int_{I_T}  \bar \vp_1\Big[(\vp_1)_t+\frac{2m(2m-2)}{(2m-1)^2}(\frac{\sn_\k'}{\sn_\k}-\frac{\sn_{4\k}'}{\sn_{4\k}})^2 \vp_1 \Big]\eta^{2m}(r)\, ds\, dt \\
 \le & -\int_{I_T}  \bar \vp_1(s,t)(\vp_1)_t(s,t)\eta^{2m}(r)\, ds\, dt\\
  =&-\frac 1 2 \int_0^{s(R)}  \bar \vp_1^2(s,T) \eta^{2m}(r) \, ds,
  \end{align*}
 where we have used inequality \eqref{ka-lm3.1-1}.
Since
$$
 \vp_1'(s,t) \bar \vp_1'(s,t)=(\bar \vp_1'(s,t))^2,
$$
then
\begin{align*}
 \int_{I_T} (\bar \vp_1'(s,t))^2 \eta^{m+2}(r)\, ds\, dt\le 0
  \end{align*}
implying $\bar\vp_1(s,t) \equiv 0$, so 
\begin{align*}
   \vp_1(s,t)\le 0
\end{align*}
for $s<s(R)$ and $t>0$. Since $\vp_1(s,t)$ satisfies  equation \eqref{ka-lm3.1-2}, then
$\vp'(s,t)<0$ for $t>0$ follows by strong maximum principle. 
\end{proof}
\begin{proof}[Proof of Theorem \ref{thm k1}]
Let $r(x)=d_M(o, x)$ be the distance function on $M$. 
By the curvature assumptions, it follows that 
\begin{align}\label{NZ-com}
\Delta r(x)\le2(m-1)\frac{\sn_\k'(r)}{\sn_\k(r)}+\frac{\sn_{4\k}'(r)}{\sn_{4\k}(r)}
\end{align}
for all $x\in M\setminus\{o, C(o)\}$ from Ni-Zheng's comparison   \cite[Theorem 1.1]{NZ18}.

We transplant Robin heat kernel $\bar H_\a(r_{\bar o},t)$ on model space $V(\bar o,\k,R)$ to geodesic ball $B_o(R)$ by
\begin{align*}
    F(x,t):=\bar H_\a (r_o(x),t), 
\end{align*}
for $x\in B_o(R)$ and $t>0$.
Clearly $F(x,t)>0$ for $t>0$ and $F(x,0)=\delta_o(x)$.
Applying inequality \eqref{NZ-com} and Lemma \ref{lm-ka}, we get
\begin{align*}
   & \p_t F(x, t)-\Delta F(x, t)\\
    =&\p_t\bar H_\a(r_o(x), t)-\bar H_\a'(r_o(x), t)\Delta r_o(x)-\bar H_\a''(r_o(x), t)\\
    \le &\p_t\bar H_\a(r_o(x), t)-\Big(2(m-1)\frac{\sn_\k'(r)}{\sn_\k(r)}+\frac{\sn_{4\k}'(r)}{\sn_{4\k}(r)}\Big)\bar H_\a'(r_o(x), t)-\bar H_\a''(r_o(x), t)\\
    =&0
\end{align*}
for all $x\in M\setminus\{o, C(o)\}$, where the last equality follows from equation \eqref{ka-1}. Using the Robin boundary condition,  we have  
\begin{align*}
\p_\nu F(x, t)  +\a  F(x, t) =\bar H_\a'(R, t)  +\a  \bar H_\a(R, t)=0
\end{align*}
for $x\in \p B_o(R) $. Since the cut locus is a null set, standard argument via approximation shows that $\bar H_\a(r_o(x),t)$ satisfies
\begin{align*}
    \begin{cases}
    \p_t F(x, t)-\Delta  F(x, t)\le 0, & (x,t)\in B_o(R)\times(0,\infty),\\
    \p_\nu  F(x, t)  +\a F(x, t) =0, & (x,t)\in \p B_o(R)\times(0,\infty),
    \end{cases}
\end{align*}
in the distributional sense. It then follows from part (2) of Proposition \ref{prop-5.1} that $ F(x, t)\le H_\a(o,x,t)$, namely
\begin{align*}
  \bar H_\a(r_o(x), t)\le H_\a(o,x,t)  
\end{align*}
for $(x,t)\in  B_o(R)\times(0,\infty)$. This completes the proof of Theorem \ref{thm k1}.
\end{proof}
For quaternion K\"ahler manifolds, we use similar arguments as in Theorem \ref{thm k1} to prove Theorem \ref{thm k2}.

\begin{proof}[Proof of Theorem \ref{thm k2}]
On quaternion K\"ahler model space, \eqref{RHK-model} gives that 
\begin{align}\label{qka-1}
   \frac{\p \bar H_\a}{\p t} -\frac{\p^2 \bar H_\a}{\p r^2} -\Big(\frac{(4m-4)\sn_\k'(r)}{\sn_\k(r)}+\frac{3\sn_{4\k}'(r)}{\sn_{4\k}(r)}\Big)\frac{\p \bar H_\a}{\p r}=0.
\end{align}
Set $\vp(s, t)=\bar H_\a(r,t)$ with
$\displaystyle
    s(r)=\int_0^ r \big(\sn_\k^{4m-4}(r)\sn^3_{4\k}(r)\big)^{\frac {1}{4m-1}}\,dr
$, 
and 
$$
\xi(r):=s'(r)=\big(\sn_\k^{4m-4}\sn^3_{4\k}\big)^{\frac {1}{4m-1}},
$$
where $r(x)=d_M(o, x)$ is the distance function of $M$. Then  equation \eqref{qka-1} is equivalent to 
\begin{align}\label{qka-1Ds}
    \p_t \vp(s,t)= \xi^2(r) \vp''(s,t)+4m\xi'(r)\vp'(s,t)
\end{align}
for $(s,t)\in (0,s(R))\times(0,\infty)$. 
Differentiating equation \eqref{qka-1Ds} in $s$  yields
\begin{align*}
    \p_t \vp'= \xi^2\vp^{(3)}+(4m+2)\xi'\vp''-\frac{4m}{4m-1}\Big((4m+8)\k+\frac{12m-12}{4m-1}(\frac{\sn_\k'}{\sn_\k}-\frac{\sn_{4\k}'}{\sn_{4\k}})^2\Big)\vp'.
\end{align*}
Then by the similar argument as in Lemma \ref{ka-1}, we obtain
\begin{align}\label{qk-d}
    \vp'(s,t)<0,
\end{align}
for $s<s(R)$ and $t>0$.  Then from  the Laplace comparison for distance function on quaternion K\"ahler manifold (see \cite[Theorem 3.2]{BY22}) 
\begin{align}\label{qk-com}
\Delta r(x)\le\frac{(4m-4)\sn_\k'(r)}{\sn_\k(r)}+\frac{3\sn_{4\k}'(r)}{\sn_{4\k}(r)},
\end{align}
and inequality  \eqref{qk-d}, 
we conclude similarly as in the proof of Theorem \ref{thm k1} that
\begin{align*}
  \bar H_\a(r_o(x), t)\le H_\a(o,x,t)  
\end{align*}
for $(x,t)\in  B_o(R)\times(0,\infty)$.  We have completed the proof of Theorem \ref{thm k2}.
\end{proof}

As  consequences of Theorem \ref{thm k1} and Theorem \ref{thm k2}, we observe the following  eigenvalue comparisons of Cheng's type
for the first Robin eigenvalue.
\begin{corollary}
Let $(M^m, g, J)$ be a  K\"ahler manifold of complex dimension $m$ whose holomorphic sectional curvature is bounded from below by $4\k$ and orthogonal Ricci curvature is bounded from below by $2(m-1)\k$ for some $\k\in \R$. 
 Let $B_{o}(R)\subset M$ be the geodesic ball of radius $R$ centered at $o$. Let $\a>0$.  Then
     \begin{align}
        \l_{1,\a} (B_o(R))\le \bar{\lambda}_1(m, \k,\a, R),
     \end{align}
where $\bar{\lambda}_1(m, \k, \a, R)$ denotes the first eigenvalue of one-dimensional eigenvalue problem \begin{equation*}
    \begin{cases}
   \vp''-\left(2(m-1)\frac{\sn_\k'(r)}{\sn_\k(r)}+\frac{\sn_{4\k}'(r)}{\sn_{4\k}(r)}\right)\vp'  =-\l\vp, \\
    \vp'(0)=0,    \quad \vp'(R)+\a \vp(R)=0.
    \end{cases}
\end{equation*}
\end{corollary}
\begin{corollary}
Let $(M^m, g, I, J, K)$ be a quaternion K\"ahler manifold of complex quaternion dimension $m$ whose scalar curvature is bounded from below by  $16m(m+2)\k$ for some $\k\in \R$. 
 Let $B_{o}(R)\subset M$ be the geodesic ball of radius $R$ centered at $o$. Let $\a>0$.  Then
     \begin{align}
        \l_{1,\a} (B_o(R))\le \bar{\lambda}_1(m, \k,\a, R),
     \end{align}
where $\bar{\lambda}_1(m, \k,\a, R)$ denotes the first eigenvalue of one-dimensional eigenvalue problem \begin{equation*}
    \begin{cases}
   \vp''-\left((4m-4)\frac{\sn_\k'(r)}{\sn_\k(r)}+\frac{3\sn_{4\k}'(r)}{\sn_{4\k}(r)}\right)\vp'  =-\l\vp, \\
    \vp'(0)=0,   \quad  \vp'(R)+\a \vp(R)=0.
    \end{cases}
\end{equation*}
\end{corollary}

\bibliographystyle{plain}
\bibliography{ref}

\begin{thebibliography}{10}

\bibitem{BY22}
Fabrice Baudoin and Guang Yang.
\newblock Brownian motions and heat kernel lower bounds on {K}\"{a}hler and
  quaternion {K}\"{a}hler manifolds.
\newblock {\em Int. Math. Res. Not. IMRN}, (6):4659--4681, 2022.

\bibitem{Ber68}
M.~Berger.
\newblock Le spectre des vari\'{e}t\'{e}s riemanniennes.
\newblock {\em Rev. Roumaine Math. Pures Appl.}, 13:915--931, 1968.

\bibitem{Cha84}
Isaac Chavel.
\newblock {\em Eigenvalues in {R}iemannian geometry}, volume 115 of {\em Pure
  and Applied Mathematics}.
\newblock Academic Press, Inc., Orlando, FL, 1984.
\newblock Including a chapter by Burton Randol, With an appendix by Jozef
  Dodziuk.

\bibitem{CY81}
Jeff Cheeger and Shing~Tung Yau.
\newblock A lower bound for the heat kernel.
\newblock {\em Comm. Pure Appl. Math.}, 34(4):465--480, 1981.

\bibitem{Cheng75}
Shiu~Yuen Cheng.
\newblock Eigenvalue comparison theorems and its geometric applications.
\newblock {\em Math. Z.}, 143(3):289--297, 1975.

\bibitem{CLY84}
Shiu~Yuen Cheng, Peter Li, and Shing-Tung Yau.
\newblock Heat equations on minimal submanifolds and their applications.
\newblock {\em Amer. J. Math.}, 106(5):1033--1065, 1984.

\bibitem{DGM77}
A.~Debiard, B.~Gaveau, and E.~Mazet.
\newblock Th\'{e}or\`emes de comparaison en g\'{e}om\'{e}trie riemannienne.
\newblock {\em Publ. Res. Inst. Math. Sci.}, 12(2):391--425, 1976/77.

\bibitem{GMN14}
Fritz Gesztesy, Marius Mitrea, and Roger Nichols.
\newblock Heat kernel bounds for elliptic partial differential operators in
  divergence form with {R}obin-type boundary conditions.
\newblock {\em J. Anal. Math.}, 122:229--287, 2014.

\bibitem{GMNO15}
Fritz Gesztesy, Marius Mitrea, Roger Nichols, and El~Maati Ouhabaz.
\newblock Heat kernel bounds for elliptic partial differential operators in
  divergence form with {R}obin-type boundary conditions {II}.
\newblock {\em Proc. Amer. Math. Soc.}, 143(4):1635--1649, 2015.

\bibitem{Gri09}
Alexander Grigor'yan.
\newblock {\em Heat kernel and analysis on manifolds}, volume~47 of {\em AMS/IP
  Studies in Advanced Mathematics}.
\newblock American Mathematical Society, Providence, RI; International Press,
  Boston, MA, 2009.

\bibitem{Li12}
Peter Li.
\newblock {\em Geometric analysis}, volume 134 of {\em Cambridge Studies in
  Advanced Mathematics}.
\newblock Cambridge University Press, Cambridge, 2012.

\bibitem{LT95}
Peter Li and Gang Tian.
\newblock On the heat kernel of the {B}ergmann metric on algebraic varieties.
\newblock {\em J. Amer. Math. Soc.}, 8(4):857--877, 1995.

\bibitem{LW21}
Xiaolong Li and Kui Wang.
\newblock First {R}obin eigenvalue of the {$p$}-{L}aplacian on {R}iemannian
  manifolds.
\newblock {\em Math. Z.}, 298(3-4):1033--1047, 2021.

\bibitem{Mar86}
Steen Markvorsen.
\newblock On the heat kernel comparison theorems for minimal submanifolds.
\newblock {\em Proc. Amer. Math. Soc.}, 97(3):479--482, 1986.

\bibitem{Mar89}
Steen Markvorsen.
\newblock On the mean exit time from a minimal submanifold.
\newblock {\em J. Differential Geom.}, 29(1):1--8, 1989.

\bibitem{Mc92}
David~M. McAvity.
\newblock Heat kernel asymptotics for mixed boundary conditions.
\newblock {\em Classical Quantum Gravity}, 9(8):1983--1997, 1992.

\bibitem{NZ18}
Lei Ni and Fangyang Zheng.
\newblock Comparison and vanishing theorems for {K}\"{a}hler manifolds.
\newblock {\em Calc. Var. Partial Differential Equations}, 57(6):Paper No. 151,
  31, 2018.

\bibitem{Palmer99}
Vicente Palmer.
\newblock Isoperimetric inequalities for extrinsic balls in minimal
  submanifolds and their applications.
\newblock {\em J. London Math. Soc. (2)}, 60(2):607--616, 1999.

\bibitem{Savo20}
Alessandro Savo.
\newblock Optimal eigenvalue estimates for the {R}obin {L}aplacian on
  {R}iemannian manifolds.
\newblock {\em J. Differential Equations}, 268(5):2280--2308, 2020.

\bibitem{Vas03}
D.~V. Vassilevich.
\newblock Heat kernel expansion: user's manual.
\newblock {\em Phys. Rep.}, 388(5-6):279--360, 2003.

\end{thebibliography}

\end{document}